\documentclass{siamart0516}

\usepackage{amsfonts}
\usepackage{graphicx}
\usepackage{epstopdf}
\usepackage{algorithmic}
\usepackage[utf8]{inputenc}
\usepackage{todonotes}

\ifpdf
  \DeclareGraphicsExtensions{.eps,.pdf,.png,.jpg}
\else
  \DeclareGraphicsExtensions{.eps}
\fi

\def\d{\mathrm{d}}

\def\e{\mathrm{e}}
\def\R{\mathbb{R}}
\def\X{\mathcal{X}}
\def\contract{\rfloor}
\DeclareMathOperator{\var}{Var}

\newcommand{\TheTitle}{The Barycenter Method for Direct Optimization} 
\newcommand{\TheAuthors}{F. Pait}

\headers{\TheTitle}{\TheAuthors}

\title{{\TheTitle}\thanks{Submitted to the editors DATE.
}}

\author{
  Felipe Pait\thanks{Universidade de São Paulo, \textsc{Brazil}
    (\email{pait@usp.br}).}
 }

\usepackage{amsopn}

\DeclareMathOperator{\tr}{Tr}

 \reversemarginpar

\ifpdf
\hypersetup{
  pdftitle={\TheTitle},
  pdfauthor={\TheAuthors}
}
\fi

\begin{document}

\maketitle

\begin{abstract} 


A randomized version of the recently developed barycenter method for derivative--free optimization  has desirable properties of a gradient search. We develop a complex version to avoid evaluations at high--gradient points. The method, applicable to non--smooth functions, is parallelizable in a natural way and shown to be robust under noisy measurements.

  
\end{abstract}



\section{The barycenter method}
I wish to search for a minimizer of a function $f : \X \rightarrow \R$ using the formula
\begin{equation}
\label{eq:barycenter}
\hat{x}_n = \frac{\sum_{i=1}^n x_i \e^{- \nu f(x_i)}}{\sum_{i=1}^n \e^{- \nu f(x_i)}},
\end{equation}
which expresses the  center of mass or barycenter of $n$ test points $x_i  \in \X \subset \R^{n_x}$ weighted according to the exponential of the  value of the function at each point. In this formula $\nu \in \R$ is a positive constant. The point $\hat{x}_n$ will be in the convex hull of the $\{x_i\}$, and we shall assume that the $n_x$--dimensional search space $\X$ is convex. Applications of the method to nonconvex search spaces are conceivable but shall not be entertained here. The rationale behind it is that points where $f$ is large receive low weight in comparison with those for which $f$ is small. 
Formula~\eqref{eq:barycenter} is suitable for {direct optimization, once scorned} \cite{once-scorned-wright}, now a respectable research area, both for its scientific challenges and practical applications. Also known as  \textsl{derivative--free optimization}, it deals with the search for extrema of a given function, employing only the values of the function and not its mathematical expression. A recent book  \cite{derivative-free-book} serves as a good entry point to the literature.
When the 1st and perhaps also 2nd derivatives of the function are available, use of gradients and Hessians leads to steepest--descent and Newton--like search algorithms. Often however derivatives are costly or impossible to compute. The challenge in direct optimization is to obtain algorithms with comparable performance, without knowledge of the derivatives. 

The barycenter method has been employed successfully to tune filters  in system identification \cite{MOLI-transactions}, see also \cite{RomanoCDC:2014}. 
The filter parameters are additional structure parameters that need to be chosen before the model parameters themselves can be optimized, and may  be said to perform a role similar to that of hyper--parameters in machine learning. With that experience in mind, the use of the barycenter method  to tune hyper--parameters seems worth exploring.
A continuous--time version of the barycenter algorithm was analyzed in \cite{pait2014reading}. Some aspects of the method and its applications were presented at SIAM conferences \cite{SiamSDiego2014,SiamBoston2016}.

A recursive version of the barycenter method follows in
\Cref{sec:recursive}. A randomized search procedure is given a probabilistic analysis in \Cref{sec:gradient}, demonstrating that it exhibits convenient gradient descent--like properties. A complex version is shown,  in \Cref{sec:complex}, to put less weight on queries made in regions where the rate of change of the goal function is large, thus discounting tests made away from candidate extrema; it may prove to be particularly  suitable for situations in which the goal function cannot be fully evaluated at each optimization step. The method's robustness  is discussed in \Cref{sec:noisy}, making a contribution to the study of optimization  under noisy measurements. \Cref{sec:paralipomena} contains applications, generalizations, connections with sundry ideas in the literature, and directions for further work. Some technical proofs are put off until \Cref{sec:tech}.


\section{Recursive algorithm}
\label{sec:recursive}

The barycenter or center of mass of a distribution of  weights $\e^{-\nu f(x_i)}$ placed at points $x_i$ is the point $x \in \R^{n_\rho}$ that minimizes the weighted sum of distances
\begin{equation*}
\label{eq:sum-of-distances}
\sum_{i=1}^n (x - x_i)^2 \e^{-\nu f(x_i)}.
\end{equation*}
Indeed, the sum is minimized when
\[\frac{\partial}{\partial x}  \sum_{i=1}^n \left. (x-x_i)^2 \e^{-\nu f(x_i)}\right|_{\hat{x}_n} 
= 2  \sum_{i=1}^n \left. (x-x_i) \e^{-\nu f(x_i)}\right|_{\hat{x}_n}  = 0,
\]
so that 
$ \hat{x}_n \sum_{i=1}^n \e^{-\nu f(x_i)} = \sum_{i=1}^n x_i \e^{-\nu f(x_i)}$, leading to \eqref{eq:barycenter}. Now writing $m_n = \sum_{i=1}^n \e^{- \nu f(x_i)}$ gives
\[
\hat{x}_n m_n = x_n  \e^{-\nu f(x_n)} + \sum_{i=1}^{n-1} x_i \e^{-\nu f(x_i)}
= x_n  \e^{-\nu f(x_n)} + m_{n-1} \hat{x}_{n-1},
\] 
from which we obtain  the recursive formulas
\begin{align}
m_n & = m_{n-1} + \e^{-\nu f(x_n)} \label{eq:mass_update} \\
\hat{x}_n & = \frac{1}{m_n} \left( m_{n-1} \hat{x}_{n-1} + \e^{-\nu f(x_n)} x_n \right). \label{eq:bary_update} 
\end{align}
Here $m_0 = 0, \hat{x}_0 $ is arbitrary, and $x_n$ is the 
 sequence of test values.

From the point of view of recursive search strategies, it can be useful to pick the sequence of test points $x_n$ as the sum of the barycenter $\hat{x}_{n-1}$ of the previous points and a  ``curiosity'' or exploration term $z_n$:
\begin{equation}
x_n  = \hat{x}_{n-1} + z_n \label{eq:periergia}.
\end{equation}
Then \eqref{eq:bary_update} reads
\begin{equation}
\hat{x}_{n} -\hat{x}_{n-1}   = 
\frac{\e^{-\nu f(x_n)}}{m_{n-1} + \e^{-\nu f(x_n)}}z_n \label{eq:periergia2}.
\end{equation}

Although the size of the steps $\hat{x}_{n} -\hat{x}_{n-1} \rightarrow 0$ as $n \rightarrow \infty$ if the function $f$ is bounded in the search set, without further assumptions on the function $f$ and on the sequence of oracle queries $x_n$
we cannot  be sure that the sequence is Cauchy.\footnote{This analysis was communicated to the author by Guilherme Vicinansa. }  
Lack of assured convergence of the barycenter $\hat{x}_{n}$  is a desirable property, because convergence should come only as a consequence of judicious choice of the sequence of tests.
%

We have not specified how the oracle queries are to be selected, and in fact the barycenter method is not in conflict with derivative--free search algorithms developed in the literature, such as Nelder--Mead's, trust--region methods, or in fact \textsl{any} combination of such strategies.  
 In \Cref{sec:gradient} a simple randomized strategy is presented and analyzed.

\section{Probabilistic analysis}
\label{sec:gradient}

{A randomized version} of the barycenter algorithm can be studied using  formula \eqref{eq:periergia2}.
If we consider $\hat{x}_n$ to be our best guess, on the basis of the information provided by the tests up to $x_n$, of where the minimum of $f(\cdot)$ might be found, then in the absence of any extra knowledge it makes sense to pick the curiosity $z_n$ as a random variable with some judiciously chosen probability distribution. 
In light of the central limit theorem, we will analyze the case where $z_n$ is  normal.

   With the goal of obtaining approximate formulas for the barycenter update rule, in the remainder of this section we  will assume  that $f(\cdot)$ is twice continuously differentiable with respect to the argument $x$. Differentiability is not required for the barycenter method to be useful, and this assumption will be weakened later on, in \Cref{sec:noisy}. 
   Define $F_n(z) = \frac{\e^{-\nu f(\hat{x}_{n-1} + z)}}{m_{n-1} + \e^{-\nu f(\hat{x}_{n-1} + z)}}$, and 
for subsequent use write
\[\bar{F}_n (z) = 
  \frac{m_{n-1} \e^{-\nu f(\hat{x}_{n-1} + z)}}{(m_{n-1} + \e^{-\nu f(\hat{x}_{n-1} + z)})^2} = 
  \frac{m_{n-1} }{m_{n-1} + \e^{-\nu f(\hat{x}_{n-1} + z)}} F_n ,\]
 so that $\frac{\partial F}{\partial z} = - \nu \bar{F} \frac{\partial f}{\partial z}.$ Here and in the computations that follow the subscript indicating dependence of the sample ordinality $n$ is omitted if there is no ambiguity. 

\begin{theorem} \label{thm:average}
If $z_n$ has a Gaussian  distribution, the expected value of $\Delta\hat{x}_n = \hat{x}_n - \hat{x}_{n-1}$ is proportional to the average value of the gradient of $f(\hat{x}_{n-1} + z_n)$ in the support of the distribution of $z$.
\end{theorem}

\begin{proof}
The  claim is established by a  calculation, which in an appropriate notation is straightforward. Consider the probability density function
\[ 
p(z)= \frac{1}{\sqrt{(2\pi)^n |\Sigma|}} e^{- \frac{1}{2}(z- \bar{z})^T\Sigma^{-1}(z- \bar{z})}.
\]
Then $\frac{\partial p}{\partial z^\beta} = - \Sigma^{-1}_{\beta \alpha } (z^\alpha- \bar{z}^\alpha) p(z) $ so $z^\alpha p =  \bar{z}^\alpha p - \Sigma^{\alpha \beta} \frac{\partial p}{\partial z^\beta}$. Einstein's implicit summation of components with equal upper and lower indices convention is in force, with upper Greek indices for the components of $z$ and of $\Sigma$, and lower indices for the components of $\Sigma$'s inverse. 

With $\mathcal X$  the $n$--dimensional set where the curiosity $z$ takes its values, for each component $z^\alpha$ of the vector $z$ we have
\begin{equation*}
E \left [ F(z) z^\alpha \right]  = \int_{\mathcal X} F(z) z^\alpha p(z) \, \d z 
= \int_{\mathcal X} F(z) \bar{z}^\alpha p(z) \, \d z 
- \Sigma^{\alpha \beta} \int_{\mathcal X} F(z)  \frac{\partial p}{\partial z^\beta}(z) \, \d z.
\end{equation*}
Using the integration--by--parts formula  in Lemma~\ref{lemma:int-by-parts},
\begin{equation*}
 \int_{\mathcal X} F(z)  \frac{\partial p}{\partial z^\beta} \, \d z 
 +  \int_{\mathcal X} \frac{\partial F}{\partial z^\beta}   p(z) \, \d z 
 =  \int_{\partial \mathcal X} F(z) p(z) \left (\frac{\partial}{\partial z^\beta} \contract \d z \right),
\end{equation*}
where $\frac{\partial}{\partial z^\beta} \contract \d z$ is an $(n-1)$--form which can be integrated at the boundary $\partial {\mathcal X}$ of the $n$--dimensional set $ {\mathcal X}$. The right--hand side is zero, because $F$ is bounded and  $p(z)$ vanishes at the border of $\mathcal X$, which is at infinity, hence
\begin{multline*}
E \left [ F(z) z^\alpha \right] 
= \bar{z}^\alpha \int_{\mathcal X} F(z)  p(z) \, \d z 
+\Sigma^{\alpha \beta} \int_{\mathcal X} \frac{\partial F}{\partial z^\beta} p(z) \, \d z\\
= E \left [ F(z)  \right] \bar{z}^\alpha 
-\nu  \Sigma^{\alpha \beta} \int_{\mathcal X} \bar{F} (z) \frac{\partial }{\partial z^\beta}\left( f(\hat{x}_{n-1} + z)\right)  p(z)\,  \d z,
\end{multline*}
so 
\begin{equation}
\label{eq:approx-grad}
\boxed
{
E \left [ \Delta\hat{x}_n \right] = E \left [ F_n (z)  \right]  \bar{z} -\nu  \Sigma \, E\left[ \bar{F_n}(z) \nabla f(\hat{x}_{n-1} + z) \right],
}
\end{equation} 
where the 2nd term is proportional to the negative gradient of $f$ as claimed. 
\end{proof}

Formula \eqref{eq:approx-grad} is a key result concerning the barycenter method. It shows that roughly speaking a random search performed in conjunction with the barycenter algorithm  follows   the direction of the negative average  gradient of the function to be minimized, the weighted average being taken over the domain where the search is performed. For a given $\nu$, the step size is essentially given by   $\Sigma$. Depending on the shape of the function $f$, large values of the variance of $z$ may compromise the descent property of the search method.

The term in $\bar{z}$ can be employed to incorporate extra knowledge in several manners. For example, if at each step we take $\bar{z} = \xi \; \Delta\hat{x}_{n-1}$, then the gradient term is responsible for the rate of change, or acceleration, of the search process. The factor $0 < \xi < 1$ is chosen to dampen oscillations and prevent instability. The case $\xi = 0$ corresponds to the garden--variety, non--accelerated gradient--like search.
In the context of intended applications to adaptive controller and filter tuning, here the author cannot resist citing his seldom--read paper \cite{acelera-letters} discussing tuners that set the 2nd derivative of the adjusted parameters, rather than the 1st derivative, as is more often done in the literature, as well as an application to filtering theory \cite{max-pait-jojoa}, where the 2nd difference is used in a discrete--time version.

The designer has the freedom to choose the free parameters $\nu$, $\Sigma$, and $\bar{z}$ of the randomized barycenter search procedure in order to achieve the most desirable convergence properties. 
{Understanding of the variance of}  $\Delta\hat{x}_n $, which depends on  the Hessian $\nabla^2 f$ of $f$, is  useful in picking these free parameters. 

\begin{theorem} \label{thm:variance}
Under the conditions of Theorem~\ref{thm:average} and assuming that the variance of $z$ is small, the variance of $\Delta\hat{x}_n$ for $\bar{z} = 0$ near a critical point of $f(x)$ where $\nabla f= 0$ is approximately 
\begin{equation}
\label{eq:approx-variance}
\mathrm{Var} (\Delta \hat{x}) \approx 
 \Sigma E[F^2 ] - 2\nu \Sigma^T E\left[  F  \bar{F} \nabla^2 f \right] \Sigma.
\end{equation}
\end{theorem}
The proof, presented in \Cref{sec:tech}, is similar to the computation of the mean. What Theorem~\ref{thm:variance}
tells us is that, near a  minimum of a locally convex function,  the variance of the adjustment step grows less than linearly with the variance of the curiosity; the higher the Hessian and the larger $\nu$ is, the smaller the variance. This is a desirable property of the method, because it indicates that the barycenter moves around less than the test points. Formula \eqref{eq:approx-variance} is only approximate, as can be confirmed by considering that the variance of $\Delta \hat{x}$  can only grow with increasing $\Sigma$ and must remain positive--definite.

 Note that by construction both functions $F$ and $\bar{F}$ are between 0 and 1.  Roughly speaking they tend to decrease together as $n$ becomes larger. If such a decrease is unwelcome, one course of action is to employ a forgetting factor as suggested in \Cref{sec:paralipomena}. To develop an intuition on what happens to the rate of adjustment $\Delta \hat{x}$ over time, it can be useful to consider a situation where the search steps are small and all variables take their mean value. In this case we have
$F_n \approx \frac{\e^{-\nu f(\hat{x}_{n-1})}}{m_{n-1} + \e^{-\nu f(\hat{x}_{n-1})}}$ and
\begin{multline*}
F_{n+1} = \frac{\e^{-\nu f(\hat{x}_{n} + z)}}{m_{n} + \e^{-\nu f(\hat{x}_{n} + z)}} 
\approx \frac{\e^{-\nu f(\hat{x}_{n-1} + \Delta \hat{x}_n)}}{m_{n-1} + \e^{-\nu f(\hat{x}_{n-1})} + \e^{-\nu f(\hat{x}_{n})}} \\
\approx \frac{\e^{-\nu f(\hat{x}_{n-1})} \e^{-\nu \nabla f(\hat{x}_{n-1}) \cdot  \Delta \hat{x}_n}}{m_{n-1} + \e^{-\nu f(\hat{x}_{n-1})} } \frac{m_n}{m_{n+1}} 
\approx F_n \frac{m_n}{m_{n+1}} (1 +  \bar{F} \nu^2 \nabla f^T \Sigma \nabla f),
\end{multline*}
where we wrote the 1st order term of the Taylor expansion of the exponential to reveal that step size tends to increase due to the gradient--descent property of the algorithm, counterbalancing the decreases as the previous weights add up. It could be interesting, in further work,  to investigate  possible connections between the probabilistic barycenter update and quasi--Newton or conjugate descent methods.

\paragraph{We have recovered a type of steepest--descent search using direct methods,\nopunct \!}
without explicitly computing derivatives or even using an expression for the function itself.
The descent  properties of the randomized barycenter method suggest that a convergence analysis along the usual lines for gradient--descent like algorithms could be endeavored, taking into account the stochastic nature of the descent. This would of course require assumptions about the properties of the specific function being optimized.
We shall not proceed along this path here. 
The analysis presented is  of a descriptive nature, showing what one might expect from the method without detailed consideration of a search strategy, rather than a prescription for a particular strategy. 

In fact we envision the barycenter method as a useful way of combining different strategies whose qualities are applicable to diverse circumstances. Some points in the sequence $x_i$ could be chosen according to coordinate--search methods, and others using simplicial or line--search ones. Alternatively, the sequence of curiosities $z_i$ could have a multimodal  distribution or be heteroscedastic, picking some points near the current barycenter with the goal of improving search precision, while perhaps a smaller fraction would have a high--variance distribution, in order to explore the possibility of distant minima smaller than the local one. 
The goals should be balanced taking into account the requirements of each problem. When the main cost of the procedure is computation of the goal function, then perhaps a fair amount of exploration may be justified by the possibility of detecting deeper minima. But in real--time design applications it may be undesirable to test points with high cost, and exploration must be used more sparingly.

\section{A stationary phase algorithm}
\label{sec:complex}

We now develop a version of the barycenter method employing complex weights.
\footnote{
We will embark on 
a digression that will  only become relevant if we try to apply the method to manifolds $\X$ whose geometry is other than Euclidian, and which can be comfortably ignored at 1st reading.
To use  complex exponents we need to identify the points $x$ not by coordinates, but by quantities that have an immediate metrical meaning. Define the positive quantity $x^\alpha$ as the distance between point $x$ and the hyperplane where the $\alpha$th coordinate is zero. Although the set of positive quantities $\{x^\alpha \}, \alpha = 1, \ldots, n_x$ is \emph{not} a vector field and does not correspond globally to a coordinate representation on the space $\mathcal X$ in any geometrically meaningful way --- in particular it does not transform correctly under changes of coordinates, each number $x^\alpha$ always remaining positive --- its numerical values coincide with the numerical values of the coordinate components of the vector $x$ in the positive orthant, which may therefore be used as local coordinates. Therefore we will take the liberty of using the habitual coordinate notation in some calculations involving the quantities $x^\alpha$.}
The \emph{complex barycenter} is defined  term--by--term by the same formula as the barycenter in \eqref{eq:barycenter}, with a complex exponent $\nu$:
 \begin{equation}
\label{eq:barycenter-i}
{\eta}^\alpha_n = \frac{\sum_{i=1}^n x^\alpha_i \e^{-\nu f(x_i)}} {\sum_{i=1}^n \e^{-\nu f(x_i)}},
\end{equation}
but now our estimate of the extremum point is
 \begin{equation}
\label{eq:barycenter-modulus}
\hat{x}^\alpha_n = | \eta_n^\alpha |.
\end{equation}
In these formulas all $x_i \geq 0$.
\footnote{The operations above are performed term--by--term, because there is no meaningful way of obtaining a coordinate representation of a point from the calculation of the modulus $| \eta | = \sqrt{\eta \bar{\eta}}$ of a complex number $\eta $. Although not a bona fide operation on vector spaces, it could be given the tensor--style expression
\begin{equation*}
\label{eq:barycenter-modulus-symbol}
\hat{x}^\alpha_i = \sqrt{\eta_i^\beta \bar{\eta}_i^\gamma \delta_{\beta \gamma}^\alpha}
\end{equation*}
using the symbol $\delta_{\beta \gamma}^\alpha = 1$ if $\alpha = \beta = \gamma$, and zero otherwise. Having said that, from now on we drop the superscripts $\alpha$ denoting coordinates unless needed.}
In the real  case the result of \eqref{eq:barycenter-i} together with \eqref{eq:barycenter-modulus} coincides with that of \eqref{eq:barycenter}.

The algorithm is suggested by Feynman's interpretation of quantum electrodynamics \cite{feynman-qed,feynman2012quantum} and by the stationary phase approximation \cite{bleistein1975asymptotic,hormander2012analysis} used in the  asymptotic analysis of integrals. 
We are interested in the situation in which $\nu$ has a real as well as an imaginary part, and, in contrast with the asymptotic expansions literature and with quantum mechanics, the imaginary part is bounded and, although perhaps large by design, not overwhelmingly so.

To understand the advantages of using a complex exponential, split the set of queries $x_i$, $i  \in \{1,2,\ldots,n\}$ into disjoint subsets $A$ and $B$, whose elements are respectively $x^A_i$, $i  \in \{1,2,\ldots,n^A\}$ and $x^B_j$, $j \in \{1,2,\ldots,n^B\}$, with $n^A + n^B = n$. Denote
${m}^A = {\sum_{i=1}^{n^A} \e^{-\nu f(x^A_i)}}$, 
${m}^B = {\sum_{j=1}^{n^B} \e^{-\nu f(x^B_j)}}$, 
${\eta}^A = {\sum_{i=1}^{n^A} x^A_i \e^{-\nu f(x^A_i)}}/m_A$, and 
${\eta}^B = {\sum_{j=1}^{n^B} x^B_j \e^{-\nu f(x^B_j)}}/m_B$. Then 
%
\begin{equation}
\label{eq:barycenter-splitAB}
{\eta}_n = \frac{\sum_{i=1}^{n^A} x^A_i \e^{-\nu f(x^A_i)}+\sum_{j=1}^{n^B} x^B_j \e^{-\nu f(x^B_j)}}
 {\sum_{i=1}^{n^A} \e^{-\nu f(xA_i)}+\sum_{j=1}^{n^B} \e^{-\nu f(x^B_j)}} 
 = \frac{m^A \eta^A +  m^B \eta^B}{m_A + m_B}.
\end{equation}
The relative weight of the measurements allocated to subset $A$ is thus determined by the sum $m^A$. The subsets $A$ and $B$ need not be organized into any temporal or other ordering; the formula is valid irrespectively of how the points $\{x_i \}$ are assigned to $A$ and $B$, and can easily be generalized to more subsets as  in \Cref{sec:paralipomena}
. 

Now suppose that all points $x_i^A$ are concentrated in a small region around a point $x^A$, so that a Taylor expansion up to the linear terms only can be used and
$f(x_i^A) \approx f(x^A) + \nabla f (x^A) \cdot y$, with $y$ in a box such that each of its coordinate components $y^\alpha \in [-\delta , \delta ], \alpha \in {1,2,\ldots,n_x}$. 
When $\nu$ is real, by convexity $\eta^A=\hat{x}^A$ will be contained inside the said region, and repeated tests would
have the effect of increasing the combined weight of the points and pulling the barycenter towards the region. If the region does not contain a candidate optimizer, this effect is unwelcome. Perhaps it can be mitigated by judicious choice of the test sequence, but often one does not have complete flexibility in specifying oracle queries, and in any case the derivatives of $f(\cdot)$ being unknown it is not possible to ascertain a priori whether a certain region contains candidate extremizers. The complex version of the barycenter method presents desirable properties in this case.

Perhaps the most transparent way to analyze this situation is to assume that the measurements are equally distributed inside the box with sides $2 \delta $. This is a reasonable practical assumption, as  there is not much reason to query the oracle more often at one point of a small  region than another.
Under these circumstances the expected value of the combined weight of the measurements $x^A$ can be approximated up to the linear term of a Taylor series
\begin{multline}
\label{eq:interference-product}
E\left [ \sum_{i=1}^{n^A} \e^{- \nu f(x_i^A)} \right ] \approx 
n^A E\left [ \e^{- \nu \left ( f(x^A)+ f'_\alpha (x^A) y^\alpha \right )} \right ]  \\
= n^A \int_{y^\alpha \in [-\delta_\alpha , \delta_\alpha ]} \e^{- \nu f(x^A)} \; \e^{-\nu f'_1 y^1}\frac{dy^1}{2 \delta_1 } \wedge \cdots \wedge \e^{-\nu f'_{n_x} y^{n_x}}\frac{dy^{n_x}}{2 \delta_{n_x} }\\
=  n^A \e^{- \nu f(x^A)} \int_{-\delta_1 }^{\delta_1 } \e^{-\nu f'_1 y^1}\frac{dy^1}{2 \delta_1 } \cdots  \int_{-\delta_{n_x} }^{\delta_{n_x} } \e^{-\nu f'_{n_x} y^{n_x}}\frac{dy^{n_x}}{2 \delta_{n_x} }.
\end{multline}
Each of the $n_x$ integrals in the product above is of the form
\begin{equation*}
\int_{-\delta }^{\delta } \e^{-\nu f'_\alpha y^\alpha} \frac{\d y^i\alpha}{2 \delta_\alpha } 
= \frac{1}{-\nu f' \delta } \frac{\e^{-\nu f'_i \delta} - \e^{\nu f'_ i \delta} }{2} 
= \frac{\sinh (\nu f' \delta )}{\nu f' \delta }.
\end{equation*}
Writing $r= \Re[\nu f' \delta ]$, $q= \Im[\nu f' \delta ]$ compute
\begin{multline*}
\sinh (\nu f' \delta ) 
 = \frac{\e^r (\cos q + j \sin q) - \e^{-r} (\cos q - j \sin q) )}{2} \\
 = \frac{(\e^r-\e^{-r}) \cos q + j (e^{r} + \e^{-r}) \sin q}{2} 
 = \sinh r \cos q + j \cosh r \sin q,
 \end{multline*}
 so that 
 \[
 \left | \frac{\sinh (\nu f' \delta )}{\nu f' \delta } \right |^2 = \frac{\sinh^2 r + \sin^2 q}{r^2 + q^2 }.
 \]
%
If we consider boxes of sides $\delta$ such that $\Im[\nu f' \delta ] = k \pi$, with $k = 1,2,\ldots$, the term $\sin^2 q = 0$. Then each of the factors in \eqref{eq:interference-product} is a complex number with magnitude $\sinh r / \sqrt{r^2 + q^2}$. Thus the expected value of the total weight of the measurements $x^A$ in a small box such that $\nabla f$ is \textsl{not} close to zero is reduced by a factor $(1 + q^2/r^2)^{n_x/2}$. It seems reasonable to choose the imaginary part of $\nu$ larger than its real part, though not so large as to cause problems with noisy measurements. We summarize these findings as follows.

\begin{theorem} \label{thm:complex-interference}
The expected contribution of measurements made outside of any region where $\nabla f \approx 0$ is discounted by one factor, proportional to $\nabla f$ and to the ratio between the complex magnitude of $\nu$ and its real part, for each dimension of the search space.
\end{theorem}

 This destructive interference, so to speak, between repeated measurements near points which are \textsl{not} candidates for minimizers is the justification for employing complex values of $\nu$.
 

For the purpose of computing $\eta_n$, the sequential ordering of the points $\{x^A_i\}$ and $\{x^B_i\}$ is immaterial. However if we were to compute $\hat{x}_i$ at a certain instant $i$ and  use its value to determine the choice of the subsequent oracle queries, along the lines suggested in the recursive algorithm analyzed in \Cref{sec:recursive}, then 
property \eqref{eq:barycenter-splitAB} would lose applicability, and the ordering of the query point $x_i$ would need to be taken into consideration. 
With some poetic license, we might say that the information responsible for the interference phenomenon, which largely cancels out the results of unproductive tests, is contained in the phase of $\eta$, and would be lost in the moment that the magnitude of $\eta$ is computed.
These considerations do not apply in the real $\nu$ case, in which $\eta_n$ and $\hat{x}_n$ coincide (for positive values of $\hat{x}_n$).

\section{Noisy measurements and non--smooth functions}
\label{sec:noisy}

Oftentimes each measurement of the function $f$ at point $x_i$ is corrupted by noise or experimental errors. In this case we still would like to minimize $f$, but now using oracle answers $f(x_i) + w_i$. For the purpose of analyzing the effect of noise on the results of the barycenter method, we consider the sequence $x_i$ as given and $w_i$ as an ergodic random process. A more elaborate analysis of the effect of noise on the sequence $\{x_i\}$ itself, which would depend on the recursive search algorithm used to generate the oracle queries, will not be attempted in this paper.

With $x_i$ deterministic and $w_i$ random, we wish to compute statistics of 
\begin{equation}
\label{eq:barycenter-noise}
{\eta}_n =  \frac{\sum_{i=1}^n x_i \e^{-\nu (f(x_i)+ w_i)}} {\sum_{i=1}^n \e^{-\nu (f(x_i)+w_i)}}.
\end{equation}
%
%
%
If $w_i$ are   independent, normal, zero--mean variables with standard deviation $\sigma$, 
then  $E[e^{-\nu w_i}] = \e^{\nu^2 \sigma^2/2} := \bar{w}$ and $\var[\e^{-\nu w_i}] = \e^{2\nu^2 \sigma^2} - \e^{\nu^2 \sigma^2} = \bar{w}^4 - \bar{w}^2$. 
Define the nominal or ``noise--free'' values 
$\bar{m} =  \sum_{i=1}^n \e^{-\nu f(x_i)}$ and 
$\bar{\eta} = { \sum_{i=1}^n x_i \e^{-\nu f(x_i)}}/ \bar{m}$, and also the scalar quantity $\bar{\bar{m}} = \sum_{i=1}^n \e^{-2\nu f(x_i)}$, the vector quantity $\bar{\bar{\eta}} = \sum_{i=1}^n x_i \e^{-2\nu f(x_i)}/ \bar{\bar{m}}$, and the matrix quantity $\breve{\eta} = \sum_{i=1}^n x_i x_i^T \e^{-2\nu f(x_i)} / \bar{\bar{m}}$.
\begin{theorem} \label{thm:under-noise}  Assuming that $\sigma$ is small, under the circumstances above the mean and variance of $\eta$ can be expressed approximately as follows:
\begin{align}
\label{eq:mean-eta}
E[\eta]  & \approx  \bar{\eta}  +
\frac{ \bar{\bar{m}} }{\bar{m}^2 } (\bar{\eta} - \bar{\bar{\eta}})  \nu^2 \sigma^2 
\text{ and } \\
\label{eq:var-eta}
\var[\eta] & \approx 
\frac{ \bar{\bar{m}} }{\bar{m}^2 }(\bar{\eta} \bar{\eta}^T   -  \bar{\eta} \bar{\bar{\eta}} - \bar{\bar{\eta}} \bar{\eta} + \breve{\eta})
 \nu^2 \sigma^2. 
\end{align}
\end{theorem}

The proof is given in \Cref{sec:tech}. What we learn from somewhat involved formulas~\eqref{eq:mean-eta}
 and \eqref{eq:var-eta} is that noise or measurement errors generate a bias and a variance in the barycenter, both proportional in 1st approximation to the variance of the noise. These  effects conspire to pull the barycenter away from the looked--after minimum of the function. However one may derive a measure of comfort in that the unwelcome errors tend to zero as the noise becomes smaller. 
Theorem~\ref{thm:under-noise} indicates that there is little reason to fear that the method breaks down under moderate measurement or computing errors.

The situation where the (unknown) function to be optimized is not smooth can be studied as a particular case of the optimization of $f(x) + w$, where $f(\cdot)$ itself is smooth, and $w(x)$ is the difference between the function under consideration and its smooth approximation, plus a noise or error parcel when applicable. In this case the assumption that $w_i$ is uncorrelated with $x_i$ can be objected to. On the other hand any function can be approximated with arbitrary precision by a smooth function. The analysis presented is  reasonable   in the practical case when   
 the approximation error is overshadowed by  measurement or  numerical errors.  

\section{Paralipomena}
\label{sec:paralipomena}
  

Formula~\eqref{eq:barycenter} is suitable for direct, or derivative--free, optimization: it is applicable when no functional expression for $f(\cdot)$ is  known, and the search for an optimizer makes use of a zero--order oracle \cite{nesterov2004introductory} which, when queried, supplies the value of the goal function at a given point, but not of its derivatives. In practical applications the oracle query may take the form of an experiment or simulation. To wrap up this paper we connect to various ideas in the literature: Laplace's method; asymmetry, which is a drawback of the method; forgetting factors; parallelization; and partial evaluation of goal functions. Then we mention applications to feedback control and dynamical system identification and offer some conclusions.

\paragraph{Laplace's method\nopunct \!}
for asymptotic analysis uses an integral equation that can be seen as a limiting case of the  barycenter method, in the limit when $\nu \rightarrow \infty$ and is real. In contrast, we are strictly interested in computable algorithms, and $\nu$ takes finite values; in fact we need to keep $\nu$ limited to stay clear of numerical troubles. For the same reason we employ a sum of a finite number of test points, which can be contrasted with Laplace's integral formula, itself more similar to the expectations appearing in the analyses in \Cref{sec:gradient} and \Cref{sec:complex}. Formula \eqref{eq:barycenter} is also related to the expression for the LogSumExp (LSE) function which is used in machine learning, and can perhaps be seen as a  way of deploying the Boltzmann distributions of statistical mechanics in an optimization procedure.

\paragraph{If the goal function is not symmetric\nopunct \!}
around the extremum, the barycenter will be pulled towards an average of the test points, which may not coincide with the actual extremum. It would be negligent not to mention this drawback of the barycenter method. In a sufficiently small neighborhood of a twice continuously differentiable function a quadratic approximation can be used, and the issue of asymmetry disappears. Higher values of $\nu$ will also tend to make the resulting bias smaller.

\paragraph{A forgetting factor\nopunct \!}
can be employed
if convergence of the step size is not desired, 
and more recent tests are to be given a bigger weight than older ones. Now in place of formulas \eqref{eq:mass_update} and \eqref{eq:bary_update}  write
\begin{equation}
\label{eq:forgetting_update}
m_n  = \lambda_n m_{n-1} + \e^{-\nu_n f(x_n)}, \qquad
\hat{x}_n  = \frac{1}{m_n} \left( \lambda_n m_{n-1} \hat{x}_{n-1} + \e^{-\nu_n f(x_n)} x_n \right).
\end{equation}
The forgetting factor $\lambda_n \leq 1$, whose dependence on $n$ allows for the case when the change happens in some instants but not all, helps the search focus closer and closer, and also to filter out noise in accumulated past measurements.
If $\lambda$ is constant then batch formula \eqref{eq:barycenter} 
needs to be replaced by
\begin{equation}
\label{eq:barycenter-forget}
\hat{x}_n = \frac{\sum_{i=1}^n x_i \lambda^{n-i} \e^{- \nu f(x_i)}}{\sum_{i=1}^n \lambda^{n-i}\e^{- \nu f(x_i)}}.
\end{equation}

If in a certain application it becomes necessary to modify step size, then another reasonable course of action is to increase $\nu$ at some, or all, instants $n$. Coherent with this strategy would be to correct $m_{n-1}$ by a factor 
$\e^{(\nu_{n}- \nu_{n-1})f(\hat{x}_{n-1})}$, to reflect the sharper weighting on latter points. No recursive formula such as \eqref{eq:forgetting_update} can be used, because the current barycenter with a given $\nu$ does not contain all the information necessary  to compute where the barycenter would have been if a different $\nu$ had been employed.

\paragraph{The barycenter method is parallelizable\nopunct \!}
 in a natural way. As equation~\eqref{eq:barycenter-splitAB} and its obvious extension 
\begin{equation} \label{eq:parallel}
{\eta} = \frac{\sum_{\ell=1}^{n_\ell} m^{A_\ell} \eta^{A_\ell}} {\sum_{\ell=1}^{n_\ell} m^{A_\ell} }
 \end{equation}
show, the search procedure may be run as a set of independent  simultaneous searches.

 The processes responsible for computing the partial barycenters of the sets $A_\ell$ may be several instances of the same algorithm, for example the type of randomized searches discussed in \Cref{sec:gradient} with different initializations or with different statistics. Such a heteroscedastic, multi--modal search can be used to probe distant corners of the search space $\X$, trying to avoid getting stuck at local minima.

While randomized local searches may be among the simplest to implement and analyze, diverse search mechanisms with complementary desirable properties and drawbacks can be incorporated and combined using \eqref{eq:parallel}. The barycenter method is not in conflict with other algorithms previously described in the literature.
Parallelization is  an effective way to take advantage of intrinsically parallel computer architectures, because the various sequences need to be combined only at the end of the search process, or perhaps at a limited number of intermediary steps. Parallelization is perhaps a compelling argument for considering the complex version of the barycenter method.

\paragraph{Partial evaluation\nopunct \!}
of the goal function is a situation that can arise if
to each query at $x_i$ the oracle returns a value $f_i(x_i)$. This situation is of practical relevance if, for example, the optimization goal can be expressed as 
\begin{equation*}
\label{eq:partial-goal-sum}
f(x) = \sum_{j=1}^k f_k(x),
\end{equation*}
with $k$ a large number (although typically $k \ll n$, the total number of oracle queries performed), which makes it impractical to evaluate the full $f(x)$ at each step. In this case our method is reminiscent of so--called stochastic gradient algorithms, although without the need for computing derivatives. 
The complex version of the barycenter method seems particularly appropriate under such circumstances. Because full information about $f(x)$ does not become available with each measurement, several oracle queries at the same or nearby points may become necessary. In hose that happen in regions of high gradient of $f(\cdot) $ are discounted as explained in \Cref{sec:complex}. 

Under partial evaluations an occasional measurement of a small value for the goal function could mislead the search procedure.
Although one should be careful not to discount the challenge of adjusting the free parameters of the  search procedure, the barycenter method seems more easily applicable in such circumstances then other derivative--free methods. It is  not readily apparent how partial--gradient methods might be used if only 0--order information on the goal function may be used.

The method developed might also be applied to situations where the goal function is partially known and can be expressed as $f(x,g(x))$, with either $f(\cdot,\cdot)$ or $g(\cdot)$ of a known form. Partial derivatives can be computed as allowed by functional knowledge, and used to guide the choice of the sequence $x_i$ according to well--established optimization methods. Those in turn can be combined with an exploration term using the barycenter method.

\paragraph{Dynamical system identification and direct adaptive control\nopunct \!}
are applications of interest  to the author where partial evaluation comes about. In system identification, one wishes to minimize an error $| \hat\theta - \theta |$ using only available measurements $|\phi(t)( \hat\theta - \theta )|$, with $\phi(t)$ a regressor that depends on the state of the dynamical system to be modeled. Often $\phi(t)$ cannot be chosen directly and the system parameter estimate $\hat\theta$ has to be derived using whatever data becomes available through observation of the process under consideration. For linear model fitting usually $\phi(t)( \hat\theta - \theta )$ is known, and gradient--type or least--squares algorithms are applicable; however there exist problems of interest such as direct adaptive control where only the error's magnitude $|\phi(t)( \hat\theta - \theta )|$ is available for the purpose of tuning $\hat\theta(t)$ \cite{design-direct}; see also \cite{miller-lyapunov}.

When optimization is  performed offline, the most important resource to conserve are time--consuming oracle queries themselves. If however optimization is used to make real--time decisions, such as those involved in the tuning of a feedback controller, then  the cost of all  queries is not equal. The costly ones are those which result in high values of the goal functions, and accordingly a more cautious type of search may be preferable.

\begin{center} * \hspace{2em} * \hspace{2em} * \end{center}

\paragraph{The goal of traditional model--based optimization\nopunct \!}
is to obtain the most precise results using the least amount of computational power. In direct optimization, the most important resource, to be used sparingly, are queries that evaluate the cost to be optimized. The barycenter method appears, from the analysis developed here, to provide a useful compromise between these goals. The computational requirements are relatively modest, as are the data requirements, and the algorithm is flexible enough to incorporate improvements if more data is available, or if there exists previous information concerning the structure of the function to be optimized.

If perhaps not completely obvious, the barycenter method has an a posteriori inevitability. The complex version in particular seems to have useful and elegant properties, which to the best of the author's knowledge have not yet been explored.

 {
 \appendix
\section{Some technical proofs}
\label{sec:tech}
\begin{proof}[Proof of Theorem~\ref{thm:variance}]
{We compute 
  the variance of } $\Delta\hat{x}_i $
 \[
 E[F^2 z^\gamma z^\alpha] -  E[F z^\gamma ] E[F  z^\alpha]
 \]
 in the easier case where $\bar{z} = 0$. Employing the expression for the derivative of $p$ as in the computation of $E[F z]$
 \[
 E[F^2 z^\gamma z^\alpha] =  \int_{\mathcal X} F^2 z^\gamma z^\alpha p \, \d z = 
 - \Sigma^{\alpha \beta} \int F^2 z^\gamma \frac{\partial p}{\partial z^\beta},
 \]
 where the notation is lightened by leaving dependence on $z$, the region of integration, and the integration form $\d z$ implicit whenever no ambiguity results. 
 Using the same integration by parts reasoning as in the proof of Theorem~\ref{thm:average},
\[
-  \int F^2 z^\gamma \frac{\partial p}{\partial z^\beta} 
=  \int  \frac{\partial }{\partial z^\beta} (F^2 z^\gamma) p 
=  \int F^2 \delta^\gamma_\beta p + \int 2 F \frac{\partial F }{\partial z^\beta}  z^\gamma p, 
\]
 where $\delta^\gamma_\beta$ is the Kronecker delta. Using the expression for the derivative of $p$   and the integration by parts formula again,  the 2nd term above in turn reads
 \[
  - 2\Sigma^{\gamma \delta}\int  F \frac{\partial F }{\partial z^\beta}  \frac{\partial p }{\partial z^\delta}
=  2\Sigma^{\gamma \delta}\int  \left(\frac{\partial  F}{\partial z^\delta}  \frac{\partial F }{\partial z^\beta} + F \frac{\partial^2 F }{\partial z^\delta \partial z^\beta}\right)p.
 \] 
Collecting the terms
 \begin{equation*}
 E[F^2 z^\gamma z^\alpha] =
  \Sigma^{\alpha \gamma} E[ F^2 ] + 2\Sigma^{\alpha \beta} \Sigma^{\gamma \delta} E\left[ \frac{\partial  F}{\partial z^\delta}  \frac{\partial F }{\partial z^\beta}
  + F  \frac{\partial^2 F }{\partial z^\delta \partial z^\beta} \right].
 \end{equation*}

 We are considering the more interesting case  when $\hat{x}$ is near a critical value of $f(\cdot)$, and $\partial F / \partial z$ is small. As
 \begin{equation*}
 E[F z^\gamma] E[F z^\alpha] =
   \Sigma^{\alpha \beta} \Sigma^{\gamma \delta} E\left[  \frac{\partial  F}{\partial z^\delta} \right] E\left[ \frac{\partial F }{\partial z^\beta}  \right],
 \end{equation*} 
 we write, ignoring the terms in $E[\partial F / \partial z]$,
 
 \begin{equation} \label{eq:variance-simplified}
  E[F^2 z z^T] -  E[F z ] E[F  z^T] \approx
  \Sigma E[F^2 ] +  2 \Sigma^T E\left[  F  \frac{\partial^2 F }{\partial z^T \partial z} \right] \Sigma.
 \end{equation}
Differentiating the equation
$
 \frac{\partial F}{\partial z^\beta} 
= -\nu   \bar{F}  \frac{\partial f}{\partial z^\beta}
$
gives
\[
 \frac{\partial^2 F}{\partial z^\alpha \partial z^\beta} 
= -\nu \bar{F} \frac{\partial^2 f}{\partial z^\alpha \partial z^\beta}
-\nu   \frac{\partial \bar{F}}{\partial z^\alpha}  \frac{\partial f}{\partial z^\beta};
\]
we can ignore the 2nd parcel which is small near a critical point, resulting in an expression for the variance
\[
\mathrm{Var} (\Delta \hat{x}) \approx 
 \Sigma E[F^2 ] - 2\nu \Sigma^T E\left[  F  \bar{F} \nabla^2 f \right] \Sigma
\] 
in terms of the Hessian $\nabla^2 f$ of $f$, which is valid approximately in a neighborhood of a critical point where $\nabla f= 0$.
\end{proof}

{A multivariable integration by parts formula}  used repeatedly in the analysis of our algorithms is derived here for completeness. 

\begin{lemma} \label{lemma:int-by-parts}
Let $\d x = \d x^1 \wedge \cdots \wedge \d x^n$ be the standard $n$-dimensional volume form in coordinates $\{x^1, \ldots x^n \}$ and let $\partial / \partial x^i$ be a coordinate vector field. With $u$ and $v$ functions defined on a manifold $\mathcal X$,
\[
\int_{\mathcal X} \left( \frac{\partial}{\partial x^i} \contract du \right) v \d x = 
\int_{\partial {\mathcal X}} uv \frac{\partial}{\partial x^i} \contract \d x    - 
\int_{\mathcal X} u \left( \frac{\partial}{\partial x^i} \contract dv \right)  \d x 
\]
\end{lemma}

\begin{proof}
As $\d x$ is an  $n$-form and $\d (uv)$ is a 1-form,
\begin{equation*}
0 = \frac{\partial}{\partial x^i} \contract \left( \d (uv) \wedge \d x \right) 
= \left( \frac{\partial}{\partial x^i} \contract  \d (uv) \right) \wedge \d x  
- d(uv) \wedge \left( \frac{\partial}{\partial x^i} \contract  \d x \right) 
\end{equation*}
but 
\begin{equation*}
\d \left(\frac{\partial}{\partial x^i} \contract \left( uv \; \d x \right) \right) 
= \d \left( uv \frac{\partial}{\partial x^i} \contract \d x \right)  
= \d (uv) \wedge \left( \frac{\partial}{\partial x^i} \contract  \d x \right) 
+ uv \; \d \left( \frac{\partial}{\partial x^i} \contract  \d x \right) 
\end{equation*}
The last $n-1$-form is closed so
\[
\left( \frac{\partial}{\partial x^i} \contract \d (uv) \right) \wedge \d x
= \d \left(\frac{\partial}{\partial x^i} \contract \left( uv \; \d x \right) \right)
\]
Using Stokes' theorem gives
\begin{equation*}
\int_{\mathcal X} \left( \frac{\partial}{\partial x^i} \contract  \d (uv) \right) \wedge \d x
= \int_{\mathcal X} \d \left(\frac{\partial}{\partial x^i} \contract \left( uv \; \d x \right) \right) 
=  \int_{\partial {\mathcal X}} \frac{\partial}{\partial x^i} \contract \left( uv \; \d x \right)
\end{equation*}
and the lemma follows from $\d (uv) = \d u \; v + u \; \d v$.
\end{proof}
 }

\begin{proof}[Proof of Theorem~\ref{thm:under-noise}]
Setting $a_i = \e^{-\nu f(x_i)}$, $b^\alpha_i = x^\alpha_i \e^{-\nu f(x_i)}$,  $v^i = \e^{-\nu w_i}$, and $\bar{v}^i= \bar{w}$
(an $n$--vector all of whose elements are $\bar{w}$), 
first gather some preliminary calculations.
%
\begin{gather*}
\frac{b_\alpha^T \bar{v}}{ a^T \bar{v}} = 
\frac{\sum_{i=1}^n x^\alpha_i \e^{-\nu f(x_i)}\bar{w}}{ \sum_{i=1}^n \e^{-\nu f(x_i)}\bar{w}} 
= \bar{\eta}^\alpha; \qquad a^T \bar{v} =   \bar{m} \bar{w};\qquad
\var[v]= (\bar{w}^4 - \bar{w}^2) I;
%
\end{gather*}
We consider $b_\alpha$ to be a column vector,  write the $\alpha$ as a subscript for typographical convenience, and $f_i$ for $f(x_i)$ to lighten notation.
\begin{multline*} 
a^T (   a b_\alpha^T  - b_\alpha a^T )\bar{v} = a^T    a  b_\alpha^T \bar{v} - a^T b_\alpha a^T \bar{v}  \\
= \sum_{i=1}^n \e^{-2\nu f_i}
\sum_{j=1}^n x^\alpha_j \e^{-\nu f_j}\bar{w}
- \sum_{i=1}^n x^\alpha_i \e^{-2\nu f_i}
\sum_{j=1}^n \e^{-\nu f_j}\bar{w}
= \bar{m}\bar{\bar{m}}(\bar{\eta} - \bar{\bar{\eta}}) \bar{w};
\end{multline*}
\begin{multline*}
\bar{v}^T (   b_\alpha a^T - ab_\alpha^T) (a b_\beta^T-b_\beta a^T) \bar{v} 
= \bar{v}^T(b_\alpha a^T a b_\beta^T   -  b_\alpha a^T b_\beta a^T 
-a b_\alpha^T a b_\beta^T  + a b_\alpha^T  b_\beta a^T)\bar{v}  \\
= \sum_{i=1}^n x^\alpha_i \e^{-\nu f_i}\bar{w} 
\sum_{j=1}^n  \e^{-2\nu f_j} 
 \sum_{k=1}^n x_k^\beta \e^{-\nu f_k}\bar{w}  
-
 \sum_{i=1}^n x^\alpha_i \e^{-\nu f_i}\bar{w} 
 \sum_{j=1}^n x^\beta_j \e^{-2\nu f_j}\bar{w} 
\sum_{k=1}^n  \e^{-\nu f_k} \\
-
\sum_{i=1}^n  \e^{-\nu f_i}
 \sum_{j=1}^n x^\alpha_j \e^{-2\nu f_j}\bar{w} 
 \sum_{k=1}^n x^\beta_k \e^{-\nu f_k}\bar{w} 
+
\sum_{i=1}^n  \e^{-\nu f_i}
 \sum_{j=1}^n x^\alpha_j x^\beta_j \e^{-2\nu f_j}\bar{w} 
\sum_{k=1}^n  \e^{-\nu f_k} \\
= \bar{\bar{m}} \bar{m}^2 (\bar{\eta}^\alpha \bar{\eta}^\beta   -  \bar{\eta}^\alpha \bar{\bar{\eta}}^\beta
-   \bar{\bar{\eta}}^\alpha  \bar{\eta}^\beta + \breve{\eta}^{\alpha \beta}) \bar{w}^2.
\end{multline*}
Lemma~\ref{lemma:mean-and-var-n}  gives 
\begin{align*}
E[\eta^\alpha] - \bar{\eta}^\alpha & \approx 
\frac{\bar{w}^4 - \bar{w}^2}{(a^T\bar{v})^3} a^T(ab^T - ba^T) \bar{v} 
= (\bar{\eta}^\alpha - \bar{\bar{\eta}}^\alpha) \frac{ \bar{\bar{m}} }{\bar{m}^2 } (\bar{w}^2-1), \\
\var[\eta]^{\alpha \beta} & \approx 
\frac{\bar{w}^4 - \bar{w}^2}{(a^T\bar{v})^4} \bar{v}^T (   b_\alpha a^T - ab_\alpha^T) I (a b_\beta^T-b_\beta a^T) \bar{v}.
\end{align*}
The last expression equals 
$(\bar{\eta}^\alpha \bar{\eta}^\beta   -  \bar{\eta}^\alpha \bar{\bar{\eta}}^\beta
-   \bar{\bar{\eta}}^\alpha  \bar{\eta}^\beta + \breve{\eta}^{\alpha \beta})
 (\bar{w}^2 -1) { \bar{\bar{m}}  }/{\bar{m}^2 }$,
and the theorem follows from the Taylor expansion
$\bar{w}^2 -1 =  \e^{\nu^2 \sigma^2} -1 \approx \nu^2 \sigma^2 + \nu^4 \sigma^4/2 + \ldots$
\end{proof}

 \begin{lemma} \label{lemma:mean-and-var-n} With $v \in \R^n$ a random positive variable such that $E[v] = \bar{v}$ and $a, b_\alpha \in \R^n$ determined, approximate values for the mean and covariance of the quotients $b_\alpha^T v / a^T v$ using only the 1st and 2nd moments of $v$ are
\begin{align}
\label{eq:mean-bvav}
E\left [ \frac{b_\alpha^T v}{a^T v} \right] & \approx  
 \frac{b_\alpha^T \bar{v}}{ a^T \bar{v}} 
 + \frac{1}{(a^T \bar{v})^3}   a^T \var[v]  (   a b_\alpha^T  - b_\alpha a^T )\bar{v} \quad \text{ and } \\
\label{eq:var-bvav}
\var \left [ \frac{b_\alpha^T v}{a^T v} \right] &  \approx
\frac{1}{(a^T \bar{v})^4}     \bar{v}^T (b_\alpha a^T - ab_\alpha^T) \var[v] (a b_\beta^T-b_\beta a^T)  \bar{v}.
\end{align}
\end{lemma}
\begin{proof}
To expand $b^T v / a^T v = (a_i v^i)^{-1} (b_j v ^j)$ for any $\alpha$ compute
\begin{align*}
\label{eq:1st2ndder}
\frac{\partial}{\partial v^k} (a_i v^i)^{-1} (b_j v^j) & = - (a_i v^i)^{-2} a_k (b_j v^j) +  (a_i v^i)^{-1} b_k \text{ and } \\
\frac{\partial^2}{\partial v^\ell v^k} (a_i v^i)^{-1} (b_j v^j) &= 2 (a_i v^i)^{-3} a_\ell a_k (b_j v^j) -  (a_i v^i)^{-2} a_k b_\ell -  (a_i v^i)^{-2} a_\ell b_k,
\end{align*}
so the first terms of the Taylor series around $\bar{v}$ read
\begin{multline}
\label{eq:taylor-bvav}
\frac{b_\alpha^T v}{ a^T v} \approx 
 \frac{b_\alpha^T \bar{v}}{ a^T \bar{v}} + \frac{\bar{v}^T}{(a^T \bar{v})^2} \left( b_\alpha a^T -  a b_\alpha^T \right) (v-\bar{v}) \\
 + \frac{1}{(a^T \bar{v})^2}(v - \bar{v})^T  \left( \frac{b_\alpha^T \bar{v}}{ a^T \bar{v}} a a^T - \frac{1}{2}(a b_\alpha^T + b_\alpha a^T) \right) (v-\bar{v}) + \ldots
\end{multline}
Taking expectation in \eqref{eq:taylor-bvav} 
\begin{equation*}
E\left [ \frac{b_\alpha^T v}{a^T v} \right]  \approx 
\frac{b_\alpha^T \bar{v}}{ a^T \bar{v}} 
 + \frac{1}{(a^T \bar{v})^2}  \tr \left( \Big( \frac{b_\alpha^T \bar{v}}{ a^T \bar{v}} a a^T - \frac{1}{2}(a b_\alpha^T + b_\alpha a^T) \Big) E[(v-\bar{v}) (v - \bar{v})^T] \right)
\end{equation*}
 which  leads to \eqref{eq:mean-bvav}. Using only the terms  up to 1st order in \eqref{eq:taylor-bvav} gives
\begin{multline*}
E\left [ \left(\frac{b_\alpha^T v}{a^T v}  - E\left [ \frac{b_\alpha^T v}{a^T v} \right ] \right)
\left(\frac{b_\beta^T v}{a^T v}  - E\left [ \frac{b_\beta^T v}{a^T v} \right
] \right) \right]  \\
\approx 
 \frac{\tr \left( (a b_\beta ^T - b_\beta a^T) \bar{v}\bar{v}^T (b_\alpha a^T - a b_\alpha^T) E[(v-\bar{v}) (v - \bar{v})^T] \right)}{(a^T \bar{v})^4},
\end{multline*}
from which \eqref{eq:var-bvav} results.
\end{proof}

\section*{Acknowledgments}
The author is grateful for the encouragement and questioning  of his students. This work was made possible by  readers like you.

\bibliographystyle{../../styles/siamart_0516/siamplain}
\bibliography{../../bibtex/bary,%
../../bibtex/acelera,%
../../bibtex/MIMOsysid}

\begin{thebibliography}{10}

\bibitem{bleistein1975asymptotic}
{\sc N.~Bleistein and R.~Handelsman}, {\em Asymptotic Expansions of Integrals},
  Dover Publications, 1975.

\bibitem{derivative-free-book}
{\sc A.~R. Conn, K.~Scheinberg, and L.~N. Vicente}, {\em Introduction to
  derivative--free optimization}, vol.~8 of MPS/SIAM Series on Optimization,
  Society for Industrial and Applied Mathematics, Philadelphia, PA, 2009.

\bibitem{feynman-qed}
{\sc R.~Feynman}, {\em {QED: the Strange Theory of Light and Matter}},
  Princeton, 1985.

\bibitem{feynman2012quantum}
{\sc R.~Feynman and A.~Hibbs}, {\em Quantum Mechanics and Path Integrals:
  Emended Edition}, Dover, 2012.

\bibitem{max-pait-jojoa}
{\sc M.~Gerken, F.~Pait, and P.~Jojoa}, {\em An adaptive filtering algorithm
  with parameter acceleration}, in IEEE International Conference on Acoustics,
  Speech, and Signal Processing, 2000, pp.~17--20.

\bibitem{hormander2012analysis}
{\sc L.~H{\"o}rmander}, {\em The Analysis of Linear Partial Differential
  Operators I: Distribution Theory and Fourier Analysis}, Grundlehren der
  mathematischen Wissenschaften, Springer Berlin Heidelberg, 2012.

\bibitem{miller-lyapunov}
{\sc D.~E. Miller and E.~J. Davison}, {\em An adaptive controller which
  provides {Lyapunov} stability}, IEEE Trans. Automatic Control, 34 (1989),
  pp.~599--609.

\bibitem{nesterov2004introductory}
{\sc Y.~Nesterov}, {\em Introductory Lectures on Convex Optimization: A Basic
  Course}, Kluwer Academic Publishers, 2004.

\bibitem{design-direct}
{\sc F.~Pait}, {\em On the design of direct adaptive controllers}, in
  Proceedings of the 40th IEEE Conference on Decision and Control, 2001,
  pp.~734--738.

\bibitem{pait2014reading}
{\sc F.~Pait}, {\em Reading {W}iener in {R}io}, in IEEE Conference on Norbert
  Wiener in the 21st Century, IEEE, 2014, pp.~1--4.

\bibitem{SiamSDiego2014}
{\sc F.~Pait and D.~Col\'on}, {\em A barycenter method for direct
  optimization}, in {SIAM} Conference on Optimization, San Diego, CA, USA, May
  2014.

\bibitem{SiamBoston2016}
{\sc F.~Pait and R.~Romano}, {\em Direct adaptive control, direct
  optimization}, in {SIAM} Annual Meeting, Boston, Massachusetts, USA, July
  2016.

\bibitem{acelera-letters}
{\sc F.~M. Pait}, {\em A tuner that accelerates parameters}, Systems \& Control
  Letters, 35 (1998), pp.~65--68.

\bibitem{RomanoCDC:2014}
{\sc R.~A. Romano and F.~Pait}, {\em Direct filter tuning and optimization in
  multivariable identification}, in Proceedings of the 53rd {IEEE} Conference
  on Decision and Control, Los Angeles, USA, Dec 2014, pp.~1798--1803.

\bibitem{MOLI-transactions}
{\sc R.~A. Romano and F.~Pait}, {\em Matchable-observable linear models and
  direct filter tuning: An approach to multivariable identification}, IEEE
  Transactions on Automatic Control, 62 (2017), pp.~2180--2193,
  \href{http://dx.doi.org/10.1109/TAC.2016.2602891}
  {doi:10.1109/TAC.2016.2602891}.

\bibitem{once-scorned-wright}
{\sc M.~Wright}, {\em {Direct search methods: Once scorned, now respectable}},
  in Numerical Analysis 1995: Proceedings of the 1995 Dundee Biennial
  Conference in Numerical Analysis, D.~F. Griffiths and G.~A. Watson, eds.,
  Addison Wesley Longman, 1995, pp.~191--208.

\end{thebibliography}

\end{document}